\documentclass[twoside,leqno,10pt, A4]{amsart}
\usepackage{amsfonts}
\usepackage{amsmath}
\usepackage{amscd}
\usepackage{amssymb}
\usepackage{amsthm}
\usepackage{amsrefs}
\usepackage{latexsym}
\usepackage{mathrsfs}
\usepackage{bbm}
\usepackage{enumerate}
\usepackage{graphicx}

\usepackage{amsfonts}
\usepackage{amsmath}
\usepackage{amscd}
\usepackage{amssymb}
\usepackage{amsthm}
\usepackage{amsrefs}
\usepackage{latexsym}
\usepackage{mathrsfs}
\usepackage{bbm}
\usepackage{amscd}
\usepackage{amssymb}
\usepackage{amsthm}
\usepackage{amsrefs}
\usepackage{latexsym}
\usepackage{mathrsfs}
\usepackage{bbm}
\usepackage{enumerate}
\usepackage{graphicx}
\usepackage{color}
\begin{document}

\newtheorem{theorem}[subsection]{Theorem}
\newtheorem{proposition}[subsection]{Proposition}
\newtheorem{lemma}[subsection]{Lemma}
\newtheorem{corollary}[subsection]{Corollary}
\newtheorem{conjecture}[subsection]{Conjecture}
\newtheorem{prop}[subsection]{Proposition}
\numberwithin{equation}{section}
\newcommand{\mr}{\ensuremath{\mathbb R}}
\newcommand{\mc}{\ensuremath{\mathbb C}}
\newcommand{\dif}{\mathrm{d}}
\newcommand{\intz}{\mathbb{Z}}
\newcommand{\ratq}{\mathbb{Q}}
\newcommand{\natn}{\mathbb{N}}
\newcommand{\comc}{\mathbb{C}}
\newcommand{\rear}{\mathbb{R}}
\newcommand{\prip}{\mathbb{P}}
\newcommand{\uph}{\mathbb{H}}
\newcommand{\fief}{\mathbb{F}}
\newcommand{\majorarc}{\mathfrak{M}}
\newcommand{\minorarc}{\mathfrak{m}}
\newcommand{\sings}{\mathfrak{S}}
\newcommand{\fA}{\ensuremath{\mathfrak A}}
\newcommand{\mn}{\ensuremath{\mathbb N}}
\newcommand{\mq}{\ensuremath{\mathbb Q}}
\newcommand{\half}{\tfrac{1}{2}}
\newcommand{\f}{f\times \chi}
\newcommand{\summ}{\mathop{{\sum}^{\star}}}
\newcommand{\chiq}{\chi \bmod q}
\newcommand{\chidb}{\chi \bmod db}
\newcommand{\chid}{\chi \bmod d}
\newcommand{\sym}{\text{sym}^2}
\newcommand{\hhalf}{\tfrac{1}{2}}
\newcommand{\sumstar}{\sideset{}{^*}\sum}
\newcommand{\sumprime}{\sideset{}{'}\sum}
\newcommand{\sumprimeprime}{\sideset{}{''}\sum}
\newcommand{\sumflat}{\sideset{}{^\flat}\sum}
\newcommand{\shortmod}{\ensuremath{\negthickspace \negthickspace \negthickspace \pmod}}
\newcommand{\V}{V\left(\frac{nm}{q^2}\right)}
\newcommand{\sumi}{\mathop{{\sum}^{\dagger}}}
\newcommand{\mz}{\ensuremath{\mathbb Z}}
\newcommand{\leg}[2]{\left(\frac{#1}{#2}\right)}
\newcommand{\muK}{\mu_{\omega}}
\newcommand{\thalf}{\tfrac12}
\newcommand{\lp}{\left(}
\newcommand{\rp}{\right)}
\newcommand{\Lam}{\Lambda_{[i]}}
\newcommand{\lam}{\lambda}
\def\L{\fracwithdelims}
\def\om{\omega}
\def\pbar{\overline{\psi}}
\def\phi{\varphi}
\def\lam{\lambda}
\def\lbar{\overline{\lambda}}
\newcommand\Sum{\Cal S}
\def\Lam{\Lambda}
\newcommand{\sumtt}{\underset{(d,2)=1}{{\sum}^*}}
\newcommand{\sumt}{\underset{(d,2)=1}{\sum \nolimits^{*}} \widetilde w\left( \frac dX \right) }
\renewcommand{\a}{\alpha}
\newcommand{\Res}{\mathop{\operatorname{Res}}}
\newcommand{\GL}{\operatorname{GL}}
\newcommand{\SL}{\operatorname{SL}}

\newcommand{\af}{\mathfrak{a}}

\theoremstyle{plain}
\newtheorem{conj}{Conjecture}
\newtheorem{remark}[subsection]{Remark}

\makeatletter
\def\widebreve{\mathpalette\wide@breve}
\def\wide@breve#1#2{\sbox\z@{$#1#2$}%
     \mathop{\vbox{\m@th\ialign{##\crcr
\kern0.08em\brevefill#1{0.8\wd\z@}\crcr\noalign{\nointerlineskip}%
                    $\hss#1#2\hss$\crcr}}}\limits}
\def\brevefill#1#2{$\m@th\sbox\tw@{$#1($}%
  \hss\resizebox{#2}{\wd\tw@}{\rotatebox[origin=c]{90}{\upshape(}}\hss$}
\makeatletter

\title[Twisted moment of $L$-functions]{Twisted first moment of quadratic and quadratic twist $L$-functions}

%%\date{\today}
\author[P. Gao]{Peng Gao}
\address{School of Mathematical Sciences, Beihang University, Beijing 100191, China}
\email{penggao@buaa.edu.cn}

\author[L. Zhao]{Liangyi Zhao}
\address{School of Mathematics and Statistics, University of New South Wales, Sydney NSW 2052, Australia}
\email{l.zhao@unsw.edu.au}

\begin{abstract}
 We evaluate the twisted first moment of central values of the product of a quadratic Dirichlet $L$-function and a quadratic twist of a modular $L$-function.
\end{abstract}

\maketitle

\noindent {\bf Mathematics Subject Classification (2010)}: 11M06, 11N37, 11F11, 11F66  \newline

\noindent {\bf Keywords}: central values, non-vanishing, first moment, quadratic Dirichlet $L$-function, modular $L$-function

\section{Introduction}
\label{sec 1}

Whether the central values of $L$-functions equal zero or not is an issue deserving and receiving much attention in the literature due to its connection with deep arithmetic problems such as the Birch and Swinnerton-Dyer conjecture.  A conjecture of S. Chowla \cite{chow} asserts that $L(1/2, \chi) \neq 0$ for every primitive Dirichlet character $\chi$. In \cite{Jutila}, M. Jutila evaluated the first and second moments of central values of the family of quadratic Dirichlet $L$-functions to show that there is an infinitude of such $L$-functions with non-vanishing central values. In \cite{sound1}, K. Soundararajan further computed the first and second mollified moments of the family of quadratic Dirichlet $L$-functions and proved that at least $87.5\%$ of these $L$-functions have non-vanishing central values. \newline

  Also using the mollifier method, B. Balasubramanian and V. K. Murty  \cite{BM} established that a small positive proportion of the family of primitive Dirichlet $L$-functions of a fixed prime modulus $q$ does not vanish result at the central point.  This proportion was subsequently improved to $1/3$ by H. Iwaniec and P. Sarnak \cite{I&S} and to $34.11\%$ by H. M. Bui \cite{Bui}, both for a general $q$.  When restricting to primes $q$, the percentage of this non-vanishing moiety was further improved to $3/8=37.5 \%$ by R. Khan and H. T. Ngo \cite{KN} and to $5/13=38.46... \%$ by R. Khan, D. Mili\'{c}evi\'{c} and H. T. Ngo \cite{KMN}. \newline

  In \cite{D&K15}, S. Das and R. Khan studied the simultaneous non-vanishing of Dirichlet $L$-functions and twists of Hecke–Maass $L$-functions of the modular group $SL_2 (\mathbb{Z})$ to show that given a Hecke-Maass cusp form $f$ of $SL_2 (\mathbb{Z})$.  They proved that there exists a primitive Dirichlet character $\chi$ modulo $q$ such that $L(1/2, f \otimes \chi)L(1/2, \overline{\chi}) \neq 0$ for any sufficiently large prime $q$.  This result was obtrained by evaluating asymptotically the average of $L(1/2, f \otimes \chi)L(1/2, \overline{\chi})$ over primitive Dirichlet characters $\chi$ modulo $q$. Further generalizations of the above result to points other than $s=1/2$ or to composite moduli $q$ can be found \cite{Sono19} and  \cite{Sun22}. \newline

Instead of studying $L$-functions twisted by characters to a fixed modulus, one may also study $L$-functions twisted by characters to a fixed order. In \cite{Maiti}, G. Maiti established a simultaneous non-vanishing result for quadratic Dirichlet $L$-functions and quadratic twists of modular $L$-functions. A key ingredient in his proof is an asymptotical formula on the first moment of central values of the product of a quadratic Dirichlet $L$-function and a quadratic twist of a modular $L$-function.  Note that rather than evaluating the moments of a given family of $L$-functions, there is currently much interest in the twisted moments of the corresponding family of $L$-functions, as the twisted moments are needed in applications when the mollifier method is used. For instance, one may apply the first and second mollified moment (when available) to achieve a positive portion simultaneous non-vanishing result concerning central values of the family of $L$-function under consideration. \newline

 It is the aim of this paper to evaluate the twisted  first moment of central values of the product of a quadratic Dirichlet $L$-function and a quadratic twist of a modular $L$-function. To state our results, we fix a holomorphic Hecke eigenform $f$ of weight $\kappa $ of $SL_2 (\mathbb{Z})$.   The Fourier expansion of $f$ at infinity can be written as
\begin{align}
\label{fFourier}
f(z) = \sum_{n=1}^{\infty} \lambda_f (n) n^{(\kappa -1)/2} e(nz), \quad \mbox{where} \quad e(z)= \exp(2 \pi i z) .
\end{align}
  Writing $\chi_d=\leg {d}{\cdot} $ for the Kronecker symbol, the twisted modular $L$-function $L(s, f \otimes \chi_d)$ is defined, for $\Re(s)>1$, by
\begin{align*}
L(s, f \otimes \chi_d) &= \sum_{n=1}^{\infty} \frac{\lambda_f(n)\chi_d(n)}{n^s}
 = \prod_{p\nmid d} \left(1 - \frac{\lambda_f (p) \chi_d(p)}{p^s}  + \frac{1}{p^{2s}}\right)^{-1}.
\end{align*}
The function $L(s, f \otimes \chi_d)$ has an analytical continuation to the entire complex plane and satisfies the functional equation (see, for example, \cite{S&Y})
\begin{align}
\label{equ:FEfd}
\Lambda (s, f \otimes \chi_d) := \left(\frac{|d|}{2\pi} \right)^s \Gamma \left( s + \frac{\kappa -1}{2} \right) L(s, f \otimes \chi_d)
= i^\kappa \epsilon(d ) \Lambda (1- s, f \otimes \chi_d),
\end{align}
where $\epsilon(d) = 1$ if $d >0$ and $\epsilon(d) = -1$ if $d<0$. \newline

Note that when $d$ is positive, odd and square-free, the quadratic character $\chi_{8d}$ is primitive with conductor $8d$ such that $\chi_{8d}(-1) = 1$. Moreover, when $\kappa \equiv 0 \pmod 4$, both of the signs of the functional equation for $L(s, f \otimes \chi_{8d})$ given in \eqref{equ:FEfd} and for $L(s, \chi_{8d})$ given in \eqref{fneqnquad} below equal to $1$.  For such $d$ and $\kappa$, we are interested in obtaining an asymptotical formula for the twisted first moment of the family $\{L(1/2, \chi_{8d})L(1/2, f \otimes \chi_{8d})\}$.  For this, we recall that the Mellin transform $\widehat f(s)$ for any function $f$ and any $s \in \mc$ is defined by
\begin{equation*}
%%\label{Phicheck}
\widehat f (s) = \int\limits_{0}^{\infty} f(x)x^{s}\frac {\dif x}{x}.
\end{equation*}

With $\zeta(s)$ denoting the Riemann zeta function, ${\mr}^+=(0,\infty)$ and $\sum^{\star}$ for the sum over square-free integers throughout, we have the following result for the twisted first moment of central values of the family of $L$-functions under our consideration.
\begin{theorem}
\label{theo:1stmoment}
 With the notation as above and $\kappa \equiv 0 \pmod 4$, let $\Phi:\mr^{+} \rightarrow \mr$ be a smooth function of compact support. For any positive odd integer $l$ such that $l=l_1l^2_2$ with $l_1$ square-free, suppose that for any $p|l_1$,
\begin{align}
\label{lcondition}
  1+\lambda_f(p)+\frac 1p \neq 0.
\end{align}
  Then we have for any $\varepsilon>0$,
\begin{align}
\label{eq:1stmoment}
\begin{split}
 \sumstar_{(d,2)=1} & L(\tfrac{1}{2}, \chi_{8d})L(\tfrac{1}{2},f \otimes \chi_{8d})\chi_{8d}(l)\Phi \left( \frac d{X} \right) \\
=&  \frac{ \widehat{\Phi}(1)X}{3\zeta(2)\sqrt{l_1}}\mathcal H(1)\prod_{p|l}Q_p(1;l)\left (\log\frac {X}{l_1}-\sum_{p \mid l_1}\log p+C+2\sum_{\substack{p|l_1 }}\frac {\lambda_f(p)\log p}{1+\lambda_f(p)+1/p}+\sum_{\substack{p| l}} \frac {D(p)\log p}{p} \right ) \\
& \hspace*{1in} +O(l^{1/4 + \varepsilon} X^{7/8 + \varepsilon}),
\end{split}
\end{align}
  where the functions $\mathcal H(s)$ and $Q_p(s;l)$ are given in Lemma \ref{lemma:DS4P} below. Moreover, $C$ is a constant whose value dependd on $\Phi(1)$ and $\Phi'(1)$ only and $D(p) \ll 1$ for all $p$.
\end{theorem}

We point out here that as already mentioned in \cite{HH22_1}, the family of $L$-functions $\{L(1/2, \chi_{8d})L(1/2, f \otimes \chi_{8d})\}$ is one of $\GL(3)$ $L$-functions, so that its first moment is comparable to the third moment of the family of quadratic Dirichlet $L$-functions studied by K. Soundarajan in \cite{sound1} and the first moment of a family of quadratic twists of $\GL(3)$ $L$-functions studied by S. Hua and B. Huang \cite{HH22_1}. \newline

  Our proof of Theorem \ref{theo:1stmoment} follows closely the methods of K. Soundararajan developed in \cite{sound1}, making also use of some treatments of M. P. Young \cite{Young2}, who gave an improvement on the error term for the smoothed third moment of quadratic Dirichlet $L$-functions. We begin with the approximate functional equation for $L(\frac{1}{2}, \chi_{8d})L(\frac{1}{2},f \otimes \chi_{8d})$ given in Lemma \ref{lem:AFE} and apply the Poisson summation formula given in Lemma \ref{lem2} to evaluate the resulting expression. In the process, we are able to identify two main terms, similar to the case of third moment of Dirichlet $L$-functions in \cite{sound1} and the case of first moment of quadratic twists of $\GL(3)$ $L$-functions in \cite{HH22_1}.  However, a salient contrast between ours and these previous results is that the two mains terms in our case are of the same size while the secondary main terms are of slightly smaller order of magnitude by some powers of $\log X$ in the earlier works \cites{sound1, HH22_1}.  This requires a more careful combination of the two main terms in our case in order to obtain the expression in \eqref{eq:1stmoment}.  We also note that as the result in Theorem \ref{theo:1stmoment} depends on the condition \eqref{lcondition}, it may not be feasible to apply the recursive method of Young developed in \cites{Young1, Young2} to ameliorate  the error term in \eqref{eq:1stmoment}. \newline

  Note that the condition \eqref{lcondition} is vacuous if $l=1$. We therefore readily recover from Theorem \ref{theo:1stmoment} the result \cite[Theorem 1.1]{Maiti} of G. Maiti on the first moment of the family of $L$-functions under consideration. With the notation as above, we have for any $\varepsilon>0$, 
\begin{align}
\label{1stmoment}
 & \sumstar_{(d,2)=1} L(\tfrac{1}{2}, \chi_{8d})L(\tfrac{1}{2},f \otimes \chi_{8d})\Phi \left( \frac d{X} \right) =  \frac{ \widehat{\Phi}(1)\mathcal H(1) X}{3\zeta(2)}\left (\log X+C \right ) +O(X^{7/8 + \varepsilon}).
\end{align}

    We point out that one may apply the result obtained in Theorem \ref{theo:1stmoment} to further achieve  sharp bounds for the $k$-th moment of $|L(\frac{1}{2}, \chi_{8d})L(\frac{1}{2},f \otimes \chi_{8d})|$ using the upper bounds principle of M. Radziwi{\l\l} and K. Soundararajan \cite{Radziwill&Sound} as well as the lower bounds principle of  W. Heap and K. Soundararajan \cite{H&Sound}. However, we notice that the asymptotic formula for the twisted first moment in \eqref{eq:1stmoment} requires the non-vanishing condition \eqref{lcondition}. In view of the Sato--Tate conjecture for modular forms (see \cite[p. 493]{iwakow}), the values of $\lambda_f(p)$ are expected to be uniformly distributed in the interval $(-2,2)$, so that one cannot rule out the possibility that the left-hand side expression in \eqref{lcondition} is close to $0$.  This, in turn, makes it challenging to combine the twisted moment result with the upper and lower bounds principles mentioned above to establish sharp bounds for the moments of  $|L(\frac{1}{2}, \chi_{8d})L(\frac{1}{2},f \otimes \chi_{8d})|$. Nevertheless, a method of K. Soundararajan \cite{Sound2009} with its refinement by A. J. Harper \cite{Harper} provides a conditional approach towards establishing upper bounds for moments of $L$-functions under the generalized Riemann hypothesis (GRH). For families of quadratic twists of $L$-functions, such sharp upper bounds are obtained by K. Sono \cite{Sono16}. It follows from \cite[Theorem 2.1]{Sono16} that for any real number $k \geq 0$, we have under GRH,
\begin{align*}
%%\label{upperbound}
   \sumstar_{(d,2)=1} |L(\tfrac{1}{2}, \chi_{8d})L(\tfrac{1}{2},f \otimes \chi_{8d})|^{2k} \Phi \left( \frac d{X} \right)  \ll_k X(\log X)^{(2k)^2}.
\end{align*}

   One may apply the above estimation together with the first moment result given in \eqref{1stmoment} to deduce a simultaneously non-vanishing result concerning $L(1/2, \chi_{8d})$ and $L(1/2,f \otimes \chi_{8d})$, using the Cauchy-Schwarz inequality. However, such a result is inferior to what one can obtain by studying the one-level density of low-lying zeros of the family of quadratic twists of $L$-functions (see for example \cite{G&Zhao12}), in which case one may deduce a positive portion non-vanishing result for $\{L(1/2, \chi_{8d})L(1/2,f \otimes \chi_{8d})\}$  under GRH. In fact, noting that $L(1/2, \chi_{8d})L(1/2,f \otimes \chi_{8d})$ is a $\GL(3)$ $L$-function, so that the one-level density result given in \cite[(1.5)]{G&Zhao12} of low-lying zeros for the family $\{L(1/2, \chi_{8d})L(1/2,f \otimes \chi_{8d})\}$  holds for any test function $\phi$ whose Fourier transform is supported on $(-2/3, 2/3)$ under GRH. This implies that one can take $v=2/3$ in \cite[(1.46)]{ILS} and therefore deduce from  \cite[(1.40)]{ILS} that at least $50\%$ of the members of the family of $\{L(1/2, \chi_{8d})L(1/2,f \otimes \chi_{8d})\}$  do not vanish under GRH. On the other hand, one checks that the above procedure does not yield a positive portion non-vanishing result unconditionally since the currently best known unconditional one-level density result given in \cite[Theorem 3.2]{Ru} is only valid for any test function $\phi$ whose Fourier transform is supported on $(-1/3, 1/3)$ for our setting. \newline

   Now, to derive a simultaneously non-vanishing result concerning $L(1/2, \chi_{8d})$ and $L(1/2,f \otimes \chi_{8d})$ unconditionally, we apply H\"older's inequality and get that
\begin{align}
\label{upperbound1}
\begin{split}
\sumstar_{(d,2)=1} & L(\tfrac{1}{2}, \chi_{8d})L(\tfrac{1}{2},f \otimes \chi_{8d})\Phi\left( \frac d{X} \right) \\
 \leq & \Big( \sumstar_{\substack{ (d,2)=1 \\ L(1/2, \chi_{8d})L(1/2, f \otimes \chi_{8d}) \neq 0}} \Phi \Big( \frac d{X} \Big) \Big)^{1/4} \Big( \sumstar_{(d,2)=1} L(\tfrac{1}{2}, \chi_{8d})^4\Phi \Big( \frac d{X} \Big) \Big)^{1/4} \Big( \sumstar_{(d,2)=1} L(\tfrac{1}{2},f \otimes \chi_{8d}) ^2\Phi \Big( \frac d{X} \Big) \Big)^{1/2}.
\end{split}
\end{align}

  We note that \cite[Theorem 1.1]{Li22} implies that
\begin{align} \label{eq:4thmoment}
    \sumstar_{\substack{(d,2)=1}} L(\tfrac{1}{2}, f \otimes \chi_{8d})^{2}\Phi \Big( \frac d{X} \Big)  \sim & X \log X.
\end{align}

  Moreover, it follows from \cite[Theorem 1]{DRHB} that
\begin{align}
\label{eq:2ndmomentest}
   \sumstar_{\substack{(d,2)=1}} L(\tfrac{1}{2}, \chi_{8d})^{4}\Phi \Big(\frac d{X} \Big)  \ll &  X^{1+\varepsilon}.
\end{align}

  We deduce readily from \eqref{upperbound1}--\eqref{eq:2ndmomentest} the following simultaneously non-vanishing result concerning $L(1/2, \chi_{8d})$ and $L(1/2,f \otimes \chi_{8d})$.
\begin{theorem}
\label{coro:nonvanish}
With the notation as above, $d$ odd, square-free and $\kappa \equiv 0 \pmod 4$, there exists an infinitude of quadratic Dirichlet characters $\chi_{8d}$ such that neither $L(1/2, \chi_{8d})$ nor $L(1/2, f \otimes \chi_{8d})$ is zero.  More precisely, the number of such characters with $d \leq X$ is $\gg X^{1-\varepsilon}$ for any $\varepsilon>0$.
\end{theorem}

    We note that the above result improves upon \cite[Corollary 1.4]{Maiti}.  Moreover, it is pointed out in \cite{Li22} that the arguments used in the proof of \cite[Theorem 1.1]{Li22} can be extended to show that
\begin{align}
\label{eq:2ndmoment}
   \sumstar_{\substack{(d,2)=1}} L(\tfrac{1}{2}, \chi_{8d})^{4}\Phi \Big( \frac d{X} \Big)  \sim &  X(\log X)^{10}.
\end{align}
   Previously, the above result was known to hold under GRH by Q. Shen \cite{Shen}.  If \eqref{eq:2ndmoment} holds and is used instead of \eqref{eq:2ndmomentest}, then the result stated in Theorem \ref{coro:nonvanish} can be improved so that the number of quadratic Dirichlet characters $\chi_{8d}$ with $d \leq X$ such that $L(1/2, \chi_{8d})L(1/2, f \otimes \chi_{8d}) \neq 0$ is $\gg X/(\log X)^{8}$.

\section{Preliminaries}
\label{sec 2}

%%----------------------------------------------------------------------------
\subsection{Sums over primes}
\label{sec2.1}
%%----------------------------------------------------------------------------

    We reserve the letter $p$ for a prime number in this paper.  Here we gather a few results concerning sums over primes.
\begin{lemma}
\label{RS} Let $x \geq 2$. We have, for some constants $b_1, b_2$,
\begin{align}
\label{merten}
\sum_{p\le x} \frac{1}{p} =& \log \log x + b_1+ O\Big(\frac{1}{\log x}\Big), \\
\label{M1}
\sum_{p\le x} \frac{\lambda^2_f(p)}{p} =& \log \log x + b_2+ O\Big(\frac{1}{\log x}\Big).
\end{align}
Moreover,
\begin{equation}
\label{mertenpartialsummation}
\sum_{p\le x} \frac {\log p}{p} = \log x + O(1).
\end{equation}
\end{lemma}
\begin{proof}
  The expressions in \eqref{merten} and \eqref{mertenpartialsummation} can be found in parts (d) and (b) of \cite[Theorem
2.7]{MVa},  respectively. The formula in \eqref{M1} follows from the Rankin--Selberg theory for $L(s, f)$, which can be found in \cite[Chapter 5]{iwakow}(see also \cite[Lemma 3]{Radziwill&Sound}).
\end{proof}

\subsection{Modular $L$-functions}
\label{sec:cusp form}

  Recall that the Fourier expansion of any holomorphic Hecke eigenform $f$ of weight $\kappa $ for the full modular group $SL_2 (\mathbb{Z})$ at infinity is given in \eqref{fFourier}. For $\Re(s) > 1$, the associated modular $L$-function $L(s, f)$ is defined by
\begin{align}
\label{Lfdef}
L(s, f) &= \sum_{n=1}^{\infty} \frac{\lambda_f (n)}{n^s}=
 \prod_{p} \left(1 - \frac{\lambda_f (p)}{p^s}  + \frac{1}{p^{2s}}\right)^{-1}= \prod_{p} \prod^2_{j =1} \left(1- \frac{\alpha_{f}(p,j)}{p^s} \right)^{-1}.
\end{align}

  The above series and product converge absolutely for $\Re(s)>1$ and $L(s, f)$ extends analytically to the entire complex
plane, satisfying the following functional equation (corresponding to the case $d=1$ of \eqref{equ:FEfd})
\begin{align*}
%%\label{equ:FEf}
\Lambda (s, f ) = \left(\frac{1}{2\pi} \right)^s \Gamma \left( s + \frac{\kappa -1}{2} \right) L(s, f )
= i^\kappa \Lambda (1- s, f ).
\end{align*}

  By Deligne's proof \cite{D} of the Weil conjecture, we know that
\begin{align}
\label{alpha}
|\alpha_{f}(p,1)|=|\alpha_{f}(p,2)|=1, \quad \alpha_{f}(p,1)\alpha_{f}(p,2)=1.
\end{align}
It follows from  this and \eqref{Lfdef} that $\lambda_f(n)$ is multiplicative and for $\nu \geq 1$, we have
\begin{align*}
%%\label{lambdafpnu}
\begin{split}
 \lambda_{f}(p^{\nu})=& \sum^{\nu}_{j=0}\alpha^{\nu-j}_f(p,1)\alpha^j_{f}(p,2).
\end{split}
\end{align*}
  The above relation further implies that $\lambda_f (n) \in \mr$, $\lambda_f (1) =1$ and $|\lambda_f (n)|
\leq d(n)$ for $n \geq 1$ with $d(n)$ denoting the number of divisors function of $n$. \newline

 Recall also that the symmetric square $L$-function $L(s, \operatorname{sym}^2 f)$ of $f$ is defined for $\Re(s)>1$ by
 (see \cite[p. 137]{iwakow} and \cite[(25.73)]{iwakow})
\begin{align*}
%%\label{Lsymexp}
\begin{split}
 L(s, \operatorname{sym}^2 f)=& \prod_p\prod_{1 \leq i \leq j \leq 2}(1-\alpha_{f}(p, i)\alpha_{f}(p, j)p^{-s})^{-1}=\prod_p (1-\alpha^2_{f}(p, 1)p^{-s})^{-1}(1-p^{-s})^{-1}(1-\alpha^2_{f}(p, 2)p^{-s})^{-1} \\
    =& \zeta(2s) \sum_{n = 1}^{\infty} \frac {\lambda_f(n^2)}{n^s}=\prod_{p} \left( 1-\frac {\lambda_f(p^2)}{p^s}+\frac {\lambda_f(p^2)}{p^{2s}}-\frac {1}{p^{3s}} \right)^{-1}.
\end{split}
\end{align*}

  It follows from a result of G. Shimura \cite{Shimura} that the corresponding completed $L$-function
\begin{align*}
%%\label{Lambdafdef}
 \Lambda(s, \operatorname{sym}^2 f)=& \pi^{-3s/2}\Gamma \left( \frac {s+1}{2} \right)\Gamma \left(\frac {s+\kappa-1}{2} \right) \Gamma \left( \frac {s+\kappa}{2} \right) L(s, \operatorname{sym}^2 f)
\end{align*}
  is entire and satisfies the functional equation
$\Lambda(s, \operatorname{sym}^2 f)=\Lambda(1-s, \operatorname{sym}^2 f)$.

%%----------------------------------------------------------------
\subsection{The approximate functional equation}
\label{sect: apprfcneqn}
%%-----------------------------------------------------------------

   Let $W(s)$ be an even entire function with rapid decay in the strip $|\Re(s)| \leq 10$ such that $W(0)=1$.  We define for $t>0$,
\begin{align}
\label{eq:Vdef}
 V(t) = \frac{1}{2 \pi i} \int\limits\limits_{(2)} \frac{W(s)}{s} w(s) t^{-s} \dif s, \quad \mbox{where} \quad w(s)= \frac {\Gamma (\frac{1}{4}+\frac s2)\Gamma (\frac{\kappa}{2}+s)}{\Gamma (\frac{1}{4})\Gamma (\frac{\kappa}{2})}\left(\frac{8}{2\pi} \right)^{s}\left(\frac{8}{\pi} \right)^{s/2}.
\end{align}
Now Stirling's formula (see \cite[(5.112)]{iwakow}) implies that for a fixed $\Re(s)$,
\begin{align}
\label{wsdecay}
 w(s) \ll e^{-|s|}.
\end{align}

  The above then implies that for any $A>0$,
\begin{align}
\label{Vest}
 V(t) \ll  \min (1, t^{-A} ).
\end{align}

   The following lemma allows us to express $L(\frac 12, \chi_{8d})L(\frac 12, f \otimes \chi_{8d})$ in terms of an essentially finite Dirichlet series.
\begin{lemma}[Approximate functional equation]
\label{lem:AFE}
 With the notation as above, we have  for positive, odd and square-free $d$ and $\kappa \equiv 0 \pmod 4$,
\begin{align*}
%%\label{fcneqnL}
  L(\tfrac 12, \chi_{8d})L(\tfrac 12, f \otimes \chi_{8d}) =& 2\sum_{n = 1}^{\infty} \frac{\chi_{8d}(n)\sigma_{f}(n)}{n^{1/2}}V \left(  \frac {n}{d^{3/2}} \right), \quad \mbox{where} \quad  \sigma_{f}(n) = \sum_{ m|n}\lambda_f(m).
\end{align*}
\end{lemma}
\begin{proof}
   Note that when $d$ is positive, odd and square-free, the $L$-function $L(s, \chi_{8d})$ satisfies the functional equation (see \cite[p. 456]{sound1}) given by
\begin{align}
\label{fneqnquad}
  \Lambda(s, \chi_{8d}) := \Big(\frac {8d} {\pi} \Big )^{s/2}\Gamma \Big( \frac {s}{2} \Big) L(s, \chi_{8d})=\Lambda(1-s,  \chi_{8d}).
\end{align}

   Similarly, for $\kappa \equiv 0 \pmod 4$, the functional equation given in \eqref{equ:FEfd} for $L(s, f \otimes \chi_{8d})$ becomes
\begin{align}
\label{equ:FE}
\Lambda (s, f \otimes \chi_{8d}) =  \left(\frac{8d}{2\pi} \right)^s \Gamma \Big( s + \frac{\kappa -1}{2} \Big) L(s, f \otimes \chi_{8d})
= \Lambda (1- s, f \otimes \chi_{8d}).
\end{align}

   It follows from \eqref{fneqnquad} and \eqref{equ:FE} that the product of the two $L$-functions $L(s, \chi_{8d})L(s, f \otimes \chi_{8d})$ satisfies the functional equation given by
\begin{align*}
%%\label{equ:FchiE}
\Lambda (s, f, d) := \Big(\frac {8d} {\pi} \Big )^{s/2}\left(\frac{8d}{2\pi} \right)^s \Gamma \Big(\frac {s}{2} \Big) \Gamma \Big(s + \frac{\kappa -1}{2}\Big) L(s, \chi_{8d})L(s, f \otimes \chi_{8d})
= \Lambda (1- s, f, d).
\end{align*}

  The assertion of the lemma now follows from the above and \cite[Theorem 5.3]{iwakow}.
\end{proof}

\subsection{Gauss sums and Poisson summation}
\label{section:Gauss}

  For all odd integers $k$ and all integers $m$, we define a Gauss-type sums $G_m(k)$, as in \cite[Sect. 2.2]{sound1},
\begin{align}
\label{G}
    G_m(k)=
    \left( \frac {1-i}{2}+\left( \frac {-1}{k} \right)\frac {1+i}{2}\right)\sum_{a \shortmod{k}}\left( \frac {a}{k} \right) e \left( \frac {am}{k} \right).
\end{align}

Let $\phi(m)$ be the Euler totient function of $m$. The next lemma is taken from \cite[Lemma 2.3]{sound1} and evaluates $G_m(k)$.
\begin{lemma}
\label{lem:Gauss}
   If $(k_1,k_2)=1$, then $G_m(k_1k_2)=G_m(k_1)G_m(k_2)$. Suppose that $p^a$ is
   the largest power of $p$ dividing $m$ (put $a=\infty$ for $m=0$).
   Then for $b \geq 1$ we have
\begin{equation*}
\label{011}
    G_m(p^b)= \left\{\begin{array}{cl}
    0  & \mbox{if $b\leq a$ is odd}, \\
    \phi(p^b) & \mbox{if $b\leq a$ is even},  \\
    -p^a  & \mbox{if $b=a+1$ is even}, \\
    (\frac {m/p^a}{p})p^a\sqrt{p}  & \mbox{if $b=a+1$ is odd}, \\
    0  & \mbox{if $b \geq a+2$}.
    \end{array}\right.
\end{equation*}
\end{lemma}

    For any smooth function $F$, define
\begin{equation} \label{tildedef}
   \widetilde{F}(\xi)=\int\limits^{\infty}_{-\infty}\left(\cos(2\pi \xi
   x)+\sin(2\pi \xi x) \right)F(x) \dif x.
\end{equation}

  We cite the following Poisson summation formula from \cite[Lemma 2.6]{sound1}.
\begin{lemma}
\label{lem2}
   Let $W$ be a smooth function compactly supported on ${\mr}^+$. We have for any odd integer $n$,
\begin{equation*}
\label{013}
  \sum_{(d,2)=1}\left( \frac {d}{n} \right)
    W\left( \frac {d}{X} \right)=\frac {X}{2n}\left( \frac {2}{n} \right)
    \sum_k(-1)^kG_k(n)\widetilde{W}\left( \frac {kX}{2n} \right),
\end{equation*}
where $G_k(n)$ is defined in \eqref{G} and $\widetilde{W}$ is defined in \eqref{tildedef}.
\end{lemma}

%%----------------------------------------------------------------------------
\subsection{Upper bounds for moments of quadratic twists of $L$-functions}
\label{sec2.4}
%%----------------------------------------------------------------------------

  In the proof of Theorem \ref{theo:1stmoment}, we need the following large sieve result for quadratic twists of $L$-functions.
\begin{lemma}
\label{lem:HB}
   With the notation above, we have, for any complex number $\sigma+it$  with $\sigma \geq 1/2$ and any $\varepsilon>0$,
\begin{align} \label{eq:Qbound}
\begin{split}
 \sumstar_{\substack{(d,2)=1 \\ d \leq X}} |L(\sigma + it, \chi_{8d})|^2 \ll_{\varepsilon} (X(1+|t|)^{1/2})^{1 + \varepsilon} \quad \mbox{and} \quad
 \sumstar_{\substack{(d,2)=1 \\  d \leq X}} |L(\sigma + it, f \otimes \chi_{8d})|^2 \ll_{\varepsilon} (X(1+|t|))^{1 + \varepsilon}.
\end{split}
\end{align}
\end{lemma}
\begin{proof}
  The first bound in \eqref{eq:Qbound} follows from \cite[Lemma 2.5]{sound1} and the second from \cite[Corollary 2.5]{S&Y}.
\end{proof}

\subsection{Analytic behaviors of some Dirichlet Series}
    Note that any odd positive integer $l$ can be written uniquely as $l=l_1l_2^2$ with $l_1, l_2$ positive and $l_1$ square-free.  For any $L$-function, we denote $L^{(c)}$ (resp. $L_{(c)}$) for the function given by the Euler product defining $L$ but omitting those primes dividing (resp. not dividing) $c$. We also write $L_p$ for $L_{(p)}$.  For example, in this notation,
\begin{align*}
%%\label{Zdef}
\begin{split}
 L_{p}(s, \sym f)= (1-p^{-s})^{-1}(1-\alpha^2_f(p,1)p^{-s})^{-1}(1-\alpha^2_f(p,2)p^{-s})^{-1}.
\end{split}
\end{align*}

Similar notations will be applied to any Dirichlet series with an Euler product.  We then define
\begin{align}
\label{Zdef}
\begin{split}
\mathcal Z(s, l)  = \sum_{\substack{m=1 \\ (m,2)=1}}^\infty\frac{\sigma_f(l_1m^2)}{m^s} (\prod_{p\mid lm}\frac{p}{p+1}).
\end{split}
\end{align}

   Our first result in this section concerns with the analytic properties of $\mathcal Z(s, l)$.
 \begin{lemma}
\label{lemma:DS4P}
With the notation as above, we have, for $\Re (s)$ large enough,
\begin{equation}
\label{eqn:DS-m}
\mathcal Z(s, l) =\prod\limits_{p\mid l}\Big ( \frac{p}{p+1} \Big )\zeta^{(2)}(s) L^{(2)}(s,\sym f)L^{(2)}(s, f)\mathcal H(s)\prod_{p|l}Q_p(s;l),
\end{equation}
where
\begin{align*}
%%\label{Hpexp}
\mathcal H(s)=\prod\limits_{(p,2)=1} \left(1 - \frac{1}{p^s} \right) \left(1 - \frac{\lambda_f (p)}{p^s}  + \frac{1}{p^{2s}}\right)H_p(s) \quad \mbox{with} \quad H_p(s)= \frac {p}{p+1}(1+\frac {1+\lambda_f(p)}{p^s})+\frac {1}{p+1}L_{p}(s, \sym f)^{-1},
\end{align*}
  and
\begin{align*}
%%\label{eqn:DS2EP}
\begin{split}
 Q_p(s;l)=& \left\{ \begin{aligned}
& \left(1+\lambda_f(p)+\frac{1}{p^s} \right)H_p(s)^{-1}, \quad \mbox{if} \quad p| l_1, \\
& \left( 1+\frac {1+\lambda_f(p)}{p^s} \right)H_p(s)^{-1}, \quad \mbox{if} \quad p| l, \ p \nmid l_1.
\end{aligned}
\right.
\end{split}
\end{align*}
 Moreover, the function $\mathcal H(s)$ is analytic in the region $\Re (s)>\frac{1}{2}$ and $|\prod\limits_{p|l}Q_p(s;l)| \ll\ l^{\varepsilon}$ if $\Re(s)> \tfrac{1}{2}+\varepsilon$.
\end{lemma}
\begin{proof}
  We note that the assertions of the lemma follow directly from \cite[Lemma 3.1]{HH22_1}, upon identifying $A(n,1)$ there with $\sigma_f(n)$ defined in this paper and observing that we have $A(1,p)=A(p,1)$ in our case. In fact, we express $\mathcal Z(s, l)$ in terms of Euler products to get that
\begin{align} \label{eqn:DS2EP}
 \mathcal Z(s, l)=\prod\limits_{p\mid l}\Big ( \frac{p}{p+1} \Big ) \prod_{\substack{ p\mid l_1}} \Big(\sum_{h=0}^\infty \frac{\sigma_f(p^{2h+1})}{p^{sh}} \Big)
\prod_{\substack{p\mid l \\ p\nmid l_1}}\Big(\sum_{h=0}^\infty \frac{\sigma_f(p^{2h})}{p^{sh}} \Big)\prod_{\substack{p\nmid 2l}}\Big(1+\sum_{h=1}^\infty \frac{\sigma_f(p^{2h})}{(p+1)p^{sh-1}}\Big).
\end{align}

  We deduce from \cite[(3.12), (3.14)]{HH22_1} that in our case
\begin{align}
\label{nongenexpZ}
\begin{split}
 \sum_{h=0}^\infty \frac{\sigma_f(p^{2h})}{p^{sh}}
=& \left(1-\frac{1}{p^s} \right)^{-1} \left( 1- \frac{\alpha^2_f(p,1)}{p^s} \right)^{-1} \left(1- \frac{\alpha^2_f(p,2)}{p^s} \right)^{-1} \left(1+\frac {1+\lambda_f(p)}{p^s}\right), \\
 \sum_{h=0}^\infty \frac{\sigma_f(p^{2h+1})}{p^{sh}}=&\frac {1+p^s(1+\lambda_f(p))}{p^s+1+\lambda_f(p)}\sum_{h=0}^\infty \frac{\sigma_f(p^{2h})}{p^{sh}} \\
=& \left(1-\frac{1}{p^s} \right)^{-1} \left( 1- \frac{\alpha^2_f(p,1)}{p^s} \right)^{-1} \left(1- \frac{\alpha^2_f(p,2)}{p^s} \right)^{-1}  \left(1+\lambda_f(p)+\frac{1}{p^s} \right).
\end{split}
\end{align}
Moreover, we also deduce from \cite[(3.13)]{HH22_1} that
\begin{align}
\label{genericexp}
\begin{split}
 1+\sum_{h=1}^\infty \frac{\sigma_f(p^{2h})}{(p+1)p^{sh-1}}
=& \left(1-\frac{1}{p^s} \right)^{-1} \left( 1- \frac{\alpha^2_f(p,1)}{p^s} \right)^{-1} \left(1- \frac{\alpha^2_f(p,2)}{p^s} \right)^{-1} H_p(s).
\end{split}
\end{align}
The expression for $\mathcal Z(s,l)$ given in \eqref{eqn:DS-m} now follows from \eqref{eqn:DS2EP}--\eqref{genericexp}.  The other assertions of the lemma follow similarly from \cite[Lemma 3.1]{HH22_1}, completing the proof of the lemma.
\end{proof}

 Let $a,l$ be two fixed co-prime odd integers. We define for $\iota =\pm 1$ and any odd, square-free $k_1$,
\begin{align}
\label{eq:Jdef}
\begin{split}
 J_{\iota k_1}(v,w)= \sum_{k_2 \geq 1} \sum_{(n, 2a)=1} \frac {\sigma_f(n)}{n^{w}k_2^{v}}  \frac {G_{\iota k_1k^2_2}(nl)}{nl}.
\end{split}
\end{align}

The following lemma develops the analytic properties of $J_{\iota k_1}(v,w)$.
\begin{lemma} \label{lemma:Jprop}
We use the same the notation as above. Suppose that $\iota = \pm 1$, $(l, 2a) =1$, $k_1$ is square-free, and $J_{\iota k_1}(v,w) $ is given initially by \eqref{eq:Jdef} for $\Re(v)$ and $\Re(w)$ large enough.  Then for any $\delta_1, \delta_2 > 0$, $J_{\iota k_1}(v,w)$ has a meromorphic continuation to
$\Re(w) > \delta_1$ and $\Re(v) \geq 1+\delta_2$.  Furthermore, in this region we have
\begin{align}
\label{Jk1def}
  J_{\iota k_1}(v,w) = \zeta(v)L^{(2)}(\thalf + w, \chi_{\iota k_1}) L^{(2)}(\thalf + w, f \otimes \chi_{\iota k_1})K_{\iota k_1}(v, w)\prod_{p|al}I_{\iota k_1,p}(v, w),
\end{align}
 where $K_{\iota k_1}(v, w)$ is analytic in this region and
\begin{align}
\label{Iest}
 \prod_{p|al}I_{\iota k_1,p}(v, w) \ll_{\delta_1, \delta_2, \varepsilon} (a|k_1|)^{\varepsilon}l^{-1/2 + \varepsilon}.
\end{align}
\end{lemma}
\begin{proof}
  Observe that $G_{\iota k_1k^2_2}(n)$ is jointly multiplicative with respect to $k_2$ and $n$ by Lemma \ref{lem:Gauss}. This allows us to write $J_{\iota k_1}(v,w)=\prod\limits_{p}J_{\iota k_1,p}(v,w)$, with $J_{\iota k_1,p}(v,w)$ defined as follows.  For $p|2a$,
\begin{align}
\label{Jpexpa}
\begin{split}
 J_{\iota k_1,p}(v,w)= \sum^{\infty}_{j=0} \frac {1}{p^{jv}}= \frac{1}{1-p^{-v}}.
\end{split}
\end{align}
If $p \nmid 2a$ with $p^{l_p} \| l$, then
\begin{align*}
%%\label{eq:MNpreDirichlet}
\begin{split}
 J_{\iota k_1,p}(v,w)= \sum^{\infty}_{n=0}\sum^{\infty}_{k_2=0} \frac {\sigma_f(p^n)}{p^{nw+k_2v}}\frac {G_{\iota k_1p^{2k_2}}(p^{n+l_p})}{p^{n+l_p}}.
\end{split}
\end{align*}

For $p \nmid k_1$, Lemma \ref{lem:Gauss} gives
\begin{align}
\label{Jgenexp}
\begin{split}
 J_{\iota k_1,p}(v,w)=& \sum_{n \equiv l_p \pmod 2}\sum_{2k_2 \geq n+l_p} \frac {\sigma_f(p^n)}{p^{nw+k_2v}}\frac {\varphi(p^{n+l_p})}{p^{n+l_p}}+\frac {\leg {\iota k_1}{p}}{\sqrt{p}}\sum_{2k_2 \geq l_p-1} \frac {\sigma_f(p^{2k_2+1-l_p})}{p^{(2k_2+1-l_p)w+k_2v}} \\
=& (1-p^{-v})^{-1} \sum_{n \equiv l_p \pmod 2} \frac {\sigma_f(p^n)}{p^{nw+(n+l_p)v/2}}\frac {\varphi(p^{n+l_p})}{p^{n+l_p}}+\frac {\leg {\iota k_1}{p}}{\sqrt{p}}\sum_{2k_2 \geq l_p-1} \frac {\sigma_f(p^{2k_2+1-l_p})}{p^{(2k_2+1-l_p)w+k_2v}}.
\end{split}
\end{align}
Similarly, Lemma \ref{lem:Gauss} also implies that if $p | k_1$,
\begin{align*}
%%\label{Jgenexp}
\begin{split}
 J_{\iota k_1,p}(v,w)
=& (1-p^{-v})^{-1} \sum_{n \equiv l_p \pmod 2} \frac {\sigma_f(p^n)}{p^{nw+(n+l_p)v/2}}\frac {\varphi(p^{n+l_p})}{p^{n+l_p}}-p^{-1} \sum_{\substack{ 2k_2+2 \geq l_p \\ k_2 \geq 0}} \frac {\sigma_f(p^{2k_2+2-l_p})}{p^{(2k_2+2-l_p)w+k_2v}}.
\end{split}
\end{align*}

  We now apply \eqref{alpha} and \eqref{nongenexpZ} to see that when $p \nmid 2ak_1l$ and $\Re(v) \geq 1+\delta_2, \Re(w) > \delta_1$, we have
\begin{align*}
%%\label{Kdef}
\begin{split}
 J_{\iota k_1,p}(v,w)
=& (1-p^{-v})^{-1} \Big (1+O(p^{-v-2w+\varepsilon})+\frac {\leg {\iota k_1}{p}}{p^{1/2+w}}(1-p^{-v})(1+\lambda_f(p)+O(p^{-v-2w+\varepsilon})) \Big) \\
=&  (1-p^{-v})^{-1} \Big ( 1+\frac {\leg {\iota k_1}{p}}{p^{1/2+w}}(1+\lambda_f(p))+O(p^{-v-2w+\varepsilon}+p^{-1/2-v-w+\varepsilon})\Big ) \\
:=&  (1-p^{-v})^{-1} \Big(1-\frac {\leg {\iota k_1}{p}}{p^{1/2+w}} \Big)^{-1} \Big(1-\frac {\leg {\iota k_1}{p}\lambda_f(p)}{p^{1/2+w}}+\frac {1}{p^{1+2w}} \Big)^{-1}K_{\iota k_1,p}(v,w).
\end{split}
\end{align*}

  We extend the above definition of $K_{\iota k_1,p}(v,w)$ to other $p$, so that by \eqref{Jgenexp}
\begin{align}
\label{Kdefgen}
\begin{split}
K_{\iota k_1,p}(v,w) = \Big(1-\frac {\leg {\iota k_1}{p}}{p^{1/2+w}} \Big)  & \Big(1-\frac {\leg {\iota k_1}{p}\lambda_f(p)}{p^{1/2+w}}+\frac {1}{p^{1+2w}} \Big) \\
& \times \Big ( \sum_{n \equiv 0 \pmod 2} \frac {\sigma_f(p^n)}{p^{nw+nv/2}}\frac {\varphi(p^{n})}{p^{n}}+(1-p^{-v})\frac {\leg {\iota k_1}{p}}{\sqrt{p}}\sum^{\infty}_{k_2 = 0} \frac {\sigma_f(p^{2k_2+1})}{p^{(2k_2+1)w+k_2v}}\Big ).
\end{split}
\end{align}

  It follows from the above that
\begin{align}
\label{Jgen}
\begin{split}
 K_{\iota k_1, p}(v,w)=1+O(p^{-1-2w}+p^{-v-2w+\varepsilon}+p^{-1/2-v-w+\varepsilon}).
\end{split}
\end{align}

  Now, consider the case $p \nmid 2al$, $p|k_1$. A short calculation shows
\begin{align}
\label{Jpdivk1}
\begin{split}
 J_{\iota k_1,p}(v,w)
=1+O(p^{-1-2w+\varepsilon}+p^{-v-2w+\varepsilon}+p^{-v+\varepsilon}).
\end{split}
\end{align}

  Next suppose that $p \nmid 2ak_1$, $p|l$. If $l_p \equiv 1 \pmod 2$, then
\begin{align}
\label{Jpl1}
\begin{split}
 J_{\iota k_1,p}(v,w)
=O(p^{-v-w+\varepsilon}+p^{-1/2+\varepsilon}).
\end{split}
\end{align}

  Similarly, if $l_p \equiv 0 \pmod 2$, then
\begin{align}
\label{Jpl2}
\begin{split}
 J_{\iota k_1,p}(v,w)
=O(p^{-v+\varepsilon}+p^{-1/2-v-w+\varepsilon}).
\end{split}
\end{align}

  Lastly, we deal with the case in which $p \nmid 2a$, $p|k_1$ and $p|l$. If $l_p \equiv 1 \pmod 2$, then
\begin{align*}
%%\label{eq:MNpreDirichlet}
\begin{split}
 J_{\iota k_1,p}(v,w)
=O(p^{-v-w+\varepsilon}+p^{-1-w+\varepsilon}).
\end{split}
\end{align*}
If $l_p \equiv 0 \pmod 2$, then
\begin{align*}
%%\label{eq:MNpreDirichlet}
\begin{split}
 J_{\iota k_1,p}(v,w) =O(p^{-v+\varepsilon}+p^{-1+\varepsilon}).
\end{split}
\end{align*}
  The above estimations imply that for all values of $l_p$, we have
\begin{align}
\label{Jpdivl}
\begin{split}
 J_{\iota k_1, p}(v,w)
\ll p^{-l_p/2}
\end{split}
\end{align}
if $\Re(v) \geq 1+\delta_2$ and $\Re(w)> \delta_1$.  Our discussions above then allow us to write $J_{\iota k_1}(v,w)$ as the right-hand side expression in \eqref{Jk1def} with
\begin{align}
\label{KIexp}
 K_{\iota k_1}(v, w) := \prod_{(p,2)=1}K_{\iota k_1,p}(v, w) \quad \mbox{and} \quad
 I_{\iota k_1,p}(v, w):=  \frac{J_{\iota k_1,p}(v,w)}{K_{\iota k_1,p}(v, w)} \Big(1-\frac {1}{p^{v}}\Big) \Big(1-\frac {\leg {\iota k_1}{p}}{p^{1/2+w}} \Big) \Big(1-\frac {\leg {\iota k_1}{p}\lambda_f(p)}{p^{1/2+w}}+\frac {1}{p^{1+2w}}\Big).
\end{align}

   It follows from \eqref{Jgen} that $K_{\iota k_1}(v, w)$  is absolutely convergent when $\Re(v) \geq 1+\delta_2, \Re(w)> \delta_1$. The estimation for $I_{\iota k_1}(v, w)$ given in \eqref{Iest} now follows from \eqref{Jgen}, \eqref{Jpdivk1} and \eqref{Jpdivl}. This completes the proof of the lemma.
\end{proof}

   The above lemma implies that $J_{\iota k_1}(v,w)$ has a simple pole at $w=1/2$ if and only if $\iota=k_1=1$, in which case we have
\begin{align}
\label{eq:Jk1pole}
  \Res\limits_{w=\frac 12}J_{1}(v,w) = \tfrac 12 \zeta (v) L^{(2)}(1, f )K_{1}(v, \tfrac 12)\prod_{p|al}I_{1,p}(v, \tfrac 12).
\end{align}

   We now define
\begin{align}
\label{Fvl}
    &   \mathcal{F}(v;l):= \sum_{\substack{(a,2l)=1}} \frac{\mu(a)}{a^{2-v}}K_{1}(v, \frac 12)\prod_{p|al}I_{1,p}(v, \frac 12).
\end{align}

  Our last result in this section develops analytic properties of some Dirichlet series related to $\mathcal{F}(v;l)$.
\begin{lemma}
\label{lemma:J1prop}
 With the notation as above, the function $K_{1}(v, 1/2)$ is analytic if
$\Re(v) \geq -1+\varepsilon$.  In this region we have
\begin{align}
\label{eq:Jk1def}
  \prod\limits_{p|al}I_{1,p}(v, \tfrac 12) \ll a^{\varepsilon}l^{-v/2+\varepsilon}.
\end{align}
  Moreover, for $-1+\varepsilon<\Re(v)<1-\varepsilon$,
\begin{align}
\label{Fdecomp}
\begin{split}
\mathcal{F}(v;l)=\mathcal{F}^{\text{gen}}(v)\prod_{p|l}E_{p}(v, \tfrac 12),
\end{split}
\end{align}
  where $\mathcal{F}^{\text{gen}}(v)$ is analytic in the region and we have
\begin{align}
\label{Edef}
\begin{split}
 E_{p}(v, \tfrac 12)=p^{-l_pv/2} \frac{L_{p}(v+1,\sym f)}{p^{l_pv/2}  \zeta_p(1)L_p(1, f) \mathcal{F}^{\text{gen}}_p(v)} \times \left\{ \begin{aligned}
& p^{-1/2+v/2} \Big( 1-\frac {1}{p^{1+v}} \Big) \Big(1+\frac {\lambda_f(p)}{p^{v}}+\frac {1}{p^{1+v}} \Big), \quad p| l_1, \\
& 1+ \frac {1+\lambda_f(p)}{p}+\frac{1}{p^{2+v}}-\frac{1}{p^{2(1+v)}}, \quad p| l, \ p \nmid l_1.
\end{aligned}
\right.
\end{split}
\end{align}

 Also, we have
\begin{align}
\label{sumovera}
\begin{split}
 \mathcal{F}(0;l)=\frac {1}{\sqrt{l_1}}\prod\limits_{p\mid l}\Big ( \frac{p}{p+1} \Big ) L^{(2)}(1,\sym f)\frac {\mathcal H(1)}{\zeta^{(2)}(2)}\prod_{p|l}Q_p(1;l).
\end{split}
\end{align}
\end{lemma}
\begin{proof}
   We keep the notations in the proof of Lemma \ref{lemma:Jprop} and deduce from  \eqref{Kdefgen} that
\begin{align}
\label{Jgenexp1}
\begin{split}
 K_{1,p}(v,\tfrac 12) =& \Big(1- \frac{1}{p^v} \Big) \Big(1-\frac{1}{p} \Big) \Big(1-\frac {\lambda_f(p)}{p}+\frac {1}{p^{2}} \Big)  J_{1,p}(v,\tfrac 12) \\
=&  \Big(1-\frac{1}{p} \Big) \Big(1-\frac {\lambda_f(p)}{p}+\frac {1}{p^{2}} \Big) \Big ( \sum_{n \equiv 0 \pmod 2} \frac {\sigma_f(p^n)}{p^{n(v+1)/2}}\frac {\varphi(p^{n})}{p^{n}}+ \Big( 1-\frac{1}{p^v} \Big) \frac {1}{p}\sum^{\infty}_{k_2 = 0} \frac {\sigma_f(p^{2k_2+1})}{p^{k_2(1+v)}} \Big ).
\end{split}
\end{align}

  We now apply \eqref{alpha} and \eqref{nongenexpZ} to evaluate the last expression in \eqref{Jgenexp1}.  This lead to
\begin{align}
\label{K1pexp}
\begin{split}
 K_{1,p}(v, \tfrac 12)= L_{p}(v+1,\sym f) \Big(1+O \Big(\frac{1}{p^2}+\frac{1}{p^{2+v}} \Big) \Big).
\end{split}
\end{align}
This readily yields the first assertion of the lemma. \newline

   We next deduce from \eqref{Jpexpa}, \eqref{KIexp} and \eqref{K1pexp} that
\begin{align*}
%%\label{KIexp0}
\begin{split}
\prod\limits_{p|al}I_{1,p}(v, \tfrac 12) \ll  (al)^{\varepsilon}\prod\limits_{p|l} \Big(1-\frac {1}{p^{v}} \Big)J_{1,p}(v,\tfrac 12).
\end{split}
\end{align*}

   Similar to the derivations of \eqref{Jpl1} and \eqref{Jpl2}, we see that when $-1+\varepsilon<v \leq 0$,
\begin{align*}
%%\label{Jpl1}
\begin{split}
 J_{1,p}(v,\tfrac 12) \Big(1-\frac {1}{p^{v}}\Big) = \left\{\begin{array}{cl}
    O(p^{-(1+l_p)v/2-1/2+\varepsilon}+p^{-(l_p-1)v/2-1/2+\varepsilon}\max (1, p^{-v})), & \mbox{if $l_p \equiv 1 \pmod 2$}, \\ \\
    O(p^{-l_pv/2+\varepsilon}+p^{-1-l_pv/2+\varepsilon}\max (1, p^{-v})), & \mbox{if $l_p \equiv 0 \pmod 2$}.
    \end{array}\right.
\end{split}
\end{align*}

 In either case, we conclude that when $v>-1+\varepsilon$,
\begin{align*}
%%\label{Jplest}
\begin{split}
 J_{1,p}(v,\tfrac 12) \Big(1-\frac {1}{p^{v}} \Big)  \ll p^{-l_pv/2+\varepsilon}.
\end{split}
\end{align*}
  The estimation given in \eqref{eq:Jk1def} now follows from this and \eqref{KIexp}.

  To establish \eqref{Fdecomp}, we first note that
\begin{align}
\label{FvEuler}
\begin{split}
 \mathcal{F}(v;l) =& K_{1}(v, \tfrac 12)\prod_{p|l}I_{1,p}(v, \tfrac 12)\prod_{(p, 2l)=1} \Big(1-\frac 1{p^{2-v}}I_{1,p}(v, \tfrac 12) \Big) \\
=& \prod_{p|l}\big (K_{1,p }(v, \tfrac 12)I_{1,p}(v, \tfrac 12)\big )\prod_{(p, 2l)=1}\Big(K_{1,p }(v, \tfrac 12)-\frac 1{p^{2-v}}I_{1,p}(v, \tfrac 12) \Big ) \\
=: & \prod_{p|l}\big (K_{1,p }(v, \tfrac 12)I_{1,p}(v, \tfrac 12)\big )\prod_{(p, 2l)=1}\mathcal{F}^{\text{gen}}_p(v),
\end{split}
\end{align}
  where
\begin{align*}
%%\label{sumovera1}
\begin{split}
 \mathcal{F}^{\text{gen}}_p(v)=\Big(1-\frac {\lambda_f(p)}{p}+\frac {1}{p^{2}} \Big) \Big(1-\frac{1}{p} \Big) \Big( \sum_{n \equiv 0 \pmod 2} \frac {\sigma_f(p^n)}{p^{n(v+1)/2}}\frac {\varphi(p^{n})}{p^{n}}+ \Big(1- \frac{1}{p^v} \Big) \frac {1}{p}\sum^{\infty}_{k_2 = 0} \frac {\sigma_f(p^{2k_2+1})}{p^{k_2(1+v)}}-\frac 1{p^{2-v}} \Big  ).
\end{split}
\end{align*}

Applying \eqref{alpha} and \eqref{nongenexpZ} to evaluate $\mathcal{F}^{\text{gen}}_p(v)$ renders
\begin{align}
\label{Fgenest}
\begin{split}
 \mathcal{F}^{\text{gen}}_p(v)=& \frac{L_{p}(v+1,\sym f)}{\zeta_p(1)L_{p}(1, f)}  \Big ( \Big(1-\frac{1}{p} \Big) \Big(1+\frac {1+\lambda_f(p)}{p^{1+v}} \Big)  \\
& \hspace*{2cm} + \Big( \frac 1p-\frac 1{p^{2-v}} \Big) L^{-1}_{p}(v+1,\sym f) + \frac {1}{p} \Big(1- \frac{1}{p^v} \Big)  \Big( 1+\lambda_f(p)+\frac{1}{p^{1+v}} \Big)\Big ) \\
=& \frac{L_{p}(v+1,\sym f)}{\zeta_p(1)L_{p}(1, f)} \Big (1+\frac {1+\lambda_f(p)}{p}+O\Big( \frac{1}{p^{2+v}}+\frac{1}{p^{2-v}} \Big)\Big ) = L_{p}(v+1,\sym f)\Big (1+O \Big( \frac{1}{p^2}+\frac{1}{p^{2+v}}+\frac{1}{p^{2-v}} \Big)\Big ).
\end{split}
\end{align}
It follows from this that $\mathcal{F}^{\text{gen}}(v):=\prod\limits_{(p, 2)=1}\mathcal{F}^{\text{gen}}_p(v)$ is analytic in the region $-1+\varepsilon<\Re(v)<1-\varepsilon$. \newline

Next \eqref{KIexp} yields
\begin{align}
\label{KIprod}
\begin{split}
 \prod_{p|l}\big (K_{1,p }(v, \tfrac 12)I_{1,p}(v, \tfrac 12)\big )=\prod_{p|l}\Big ( \Big(1-\frac {\lambda_f(p)}{p}+\frac {1}{p^{2}} \Big) \Big(1-\frac 1p \Big) \Big(1-\frac 1{p^v})J_{1,p}(v,\tfrac 12 \Big)\Big ).
\end{split}
\end{align}

Now \eqref{Jgenexp} gives that if $l_p \equiv 1 \pmod 2$, then
\begin{align}
\label{KIprodevl}
\begin{split}
 & \Big(1-\frac 1{p^v}\Big)J_{1,p}(v,\tfrac 12)
=  p^{-1/2-(l_p-1)v/2}L_{p}(v+1,\sym f) \Big(1-\frac {1}{p^{1+v}} \Big) \Big(1+\frac {\lambda_f(p)}{p^{v}}+\frac {1}{p^{1+v}} \Big).
\end{split}
\end{align}

  Also, if $l_p \equiv 0 \pmod 2$, we have
\begin{align}
\label{KIprodevl2} \begin{split}
 \Big(1-\frac 1{p^v} \Big)J_{1,p}(v,\tfrac 12)
=& p^{-l_pv/2}L_{p}(v+1,\sym f) \Big(1+ \frac {1+\lambda_f(p)}{p}+\frac{1}{p^{2+v}}-\frac{1}{p^{2(1+v)}} \Big).
\end{split}
\end{align}

  We readily deduce the expression given in \eqref{Edef} for $E_p(v,\frac 12)$ as well as the expression for $\mathcal F(0;l)$ given in \eqref{sumovera} from \eqref{FvEuler}--\eqref{KIprodevl2}, via direct computations. This completes the proof of the lemma.
\end{proof}

\section{Proof of Theorem \ref{theo:1stmoment}}
\label{sec 3}

\subsection{A first decomposition}
\label{section:mainprop}

From Lemma \ref{lem:AFE}, we get
\[
\sumstar_{(d,2)=1} L(\tfrac{1}{2}, \chi_{8d})L(\tfrac{1}{2},f \otimes \chi_{8d})\chi_{8d}(l) = 2\sumstar_{(d,2)=1} \sum_{n} \frac { \chi_{8d}(nl)\sigma_f(n)}{n^{1/2}}\Phi \left(\frac {d}{X} \right )V\Big( \frac n{d^{3/2}} \Big):=S(l).
\]

We apply the M\"{o}bius inversion to remove the square-free condition on $d$ to obtain that, for an appropriate parameter $Z$ to be optimized later,
\begin{eqnarray*}
 S(l) = \Big(\sum_{\substack{a \leq Z \\ (a,2l)=1}} + \sum_{\substack{a > Z \\ (a,2l)=1}} \Big) 2 \mu(a)  \sum_{(d,2)=1} \sum_{(n,a)=1} \frac { \chi_{8d}(nl)\sigma_f(n)}{n^{1/2}}\Phi \left(\frac {da^2}{X} \right )V \Big( \frac n{(a^2d)^{3/2}} \Big) : =S_1(l)+ S_2(l).
\end{eqnarray*}

\subsection{Estimating $S_2(l)$}
\label{section:S2}

  We estimate $S_2(l)$ by writing $d=b^2 \ell$ with $\ell$ being square-free, and
group terms according to $c=ab$ to deduce that
\begin{equation*}
%%\label{eq:S21}
 S_2(l) = 2\sum_{(c,2l)=1} \sum_{\substack{a > Z \\ a|c}} \mu(a)
 \sumstar_{\ell} \sum_{(n,c)=1}
  \frac { \chi_{8\ell}(nl)\sigma_f(n)}{n^{1/2}} \Phi \left(\frac {\ell c^2}{X} \right )V \Big(\frac n{(c^2\ell)^{3/2}} \Big).
\end{equation*}

Inserting the definition of $V$ in \eqref{eq:Vdef}, we arrive at
\begin{align} \label{eq:S22}
\begin{split}
 S_2(l) =& 2\sum_{(c,2l)=1} \sum_{\substack{a > Z \\ a|c}} \mu(a)
 \sumstar_{\ell}  \frac{1}{2 \pi i} \int\limits\limits_{(2)} \frac{W(s)}{s} w(s) \sum_{(n,c)=1}
  \frac { \chi_{8\ell}(nl)\sigma_f(n)}{n^{1/2+s}} \Phi \left(\frac {\ell c^2}{X} \right )(c^2\ell)^{3s/2} \dif s \\
=& 2\sum_{(c,2l)=1} \sum_{\substack{a > Z \\ a|c}} \mu(a)
 \sumstar_{\ell}  \frac{\chi_{8\ell}(l)}{2 \pi i} \int\limits\limits_{(2)} \frac{W(s)}{s} w(s)L^{(c)}(\tfrac 12+s, \chi_{8\ell})L^{(c)}(\tfrac 12+s, f \otimes \chi_{8\ell}) \Phi \left(\frac {\ell c^2}{X} \right )(c^2\ell)^{3s/2} \dif s.
\end{split}
\end{align}

 We now move the line of integration in \eqref{eq:S22} to $\Re(s)=1/\log X$ without encountering any pole.  On the new line of integration, we have
\begin{align}
\label{Linversebound}
L^{(c)}(\tfrac 12+s, \chi_{8\ell})L^{(c)}(\tfrac 12+s, f \otimes \chi_{8\ell})|
\le d(c)^2 ( |L(\tfrac 12+s, \chi_{8\ell})|^2 + |L(\tfrac 12+s, f \otimes \chi_{8\ell})|^2),
\end{align}
 where we recall that $d(c)$ denotes the number of divisors of $c$. \newline

  We apply \eqref{Linversebound}  in \eqref{eq:S22} and make use of Lemma \ref{lem:HB} to produce the bound
\begin{align} \label{S2}
  S_2(l) \ll \sum_{c} \sum_{\substack{a > Z \\ a|c}}   d(c)^2 \int\limits_{-\infty}^{\infty} (1+|t|)^{-4} \sumstar_{\ell\ll X/c^2}
\Big (\Big|L(\tfrac 12+\tfrac 1{\log X}+it,  \chi_{8\ell})\Big|^{2} + \Big|L(\tfrac 12+\tfrac 1{\log X}+it,  f \otimes \chi_{8\ell})\Big|^{2} \Big ) \ \dif t \ll  \frac {X^{1+\varepsilon}}{Z}.
\end{align}

\subsection{Estimating $S_1(l)$, the main term}
Now, we apply the Poisson summation formula in Lemma \ref{lem2} to deduce that
\begin{align}
\label{eq:S1}
 S_1(l)= X \sum_{n} \frac {\sigma_f(n)}{n^{1/2}} \frac {1}{nl}\left( \frac {16}{nl} \right)
   \sum_{\substack{a \leq Z \\ (a,2nl)=1}} \frac{\mu(a)}{a^2} \sum_k(-1)^kG_k(nl)\widetilde{\Phi}_n\left( \frac {kX}{2a^2 nl} \right), \quad \mbox{where} \quad   \Phi_y(t) =\Phi(t)V \Big( \frac y{(Xt)^{3/2}} \Big).
\end{align}

We write $S_{10}(l)$ for the $k=0$ term in \eqref{eq:S1}. Note that
$$
\sum_{\substack{a \leq Z \\ (a,2nl)=1}} \frac{\mu(a)}{a^2} = \frac{1}{\zeta(2)}
\prod_{p|2nl} \left( 1-\frac{1}{p^2}\right)^{-1} +O(Z^{-1}) = \frac{1}{\zeta^{(2)}(2)}\prod_{p|nl} \left(1-\frac{1}{p^2}\right)^{-1}+ O \Big( \frac{1}{Z} \Big) .
$$
 Moreover, we note that Lemma \ref{lem:Gauss} implies that $G_0(ln)=\phi(ln)$ if $ln=\Box$ and $G_0(ln)=0$ otherwise. Recall that $l=l_1l_2^2$ with $l_1$ square-free.  The condition $ln=\Box$ is thus equivalent to $n=l_1m^2$ for some integer $m$. It follows from the above discussions and \eqref{eq:S1} that
\begin{align*}
%%\label{eq:S1}
 S_{10}(l)= \frac{X}{\zeta^{(2)}(2)\sqrt{l_1}}\sum_{\substack{m=1 \\ (m,2)=1 }}^\infty\frac{\sigma_f(l_1m^2)}{m} \widetilde{\Phi}_{l_1m^2}(0) \prod_{p\mid lm}\frac{p}{p+1} +O\Big( \frac{X}{Z\sqrt{l_1}}\sum_{\substack{m=1 \\ (m,2)=1 }}^\infty\frac{|\sigma_f(l_1m^2)|}{m}|\widetilde{\Phi}_{l_1m^2}(0)| \Big).
\end{align*}

  Notice that as $\Phi(t)$ is supported in $\mr^+$, we have for any $c>0$,
\begin{align} \label{phi0}
\begin{split}
 \widetilde{\Phi}_{l_1m^2}(0)&=\int\limits_0^\infty \Phi(t) V \Big(\frac{l_1m^2}{(Xt)^{3/2}} \Big) \dif t =\frac{1}{2\pi i}\int\limits_{(c)}w(s) \Big(\frac{X^{3/2}}{l_1m^2}\Big)^s \Big(\int\limits_0^\infty\Phi(t)t^{3s/2} \dif t \Big)\frac{W(s) \dif s}{s}\\
&=\frac{1}{2\pi i}\int\limits_{(c)}w(s)\Big( \frac{X^{3/2}}{l_1m^2} \Big)^s \widehat{\Phi}(\tfrac{3s}{2}+1) \frac{ W(s) \dif s}{s}.
\end{split}
\end{align}

 We apply the first equality above and \eqref{Vest} which implies that we may assume that $l_1m^2 \ll X^{3/2+\varepsilon}$, so that
\begin{align*}
\label{eqn:Perror}
\begin{split}
\sum_{\substack{m=1 \\ (m,2)=1 }}^\infty\frac{|\sigma_f(l_1m^2)|}{m}|\widetilde{\Phi}_{l_1m^2}(0)|&\ll \sum_{m^2\ll X^{3/2+\varepsilon}/l_1}\frac{|\sigma_f(l_1m^2)|}{m}\ll \sum_{m^2\ll X^{3/2+\varepsilon}/l_1}\frac{(l_1m^2)^{\varepsilon}}{m}
\ll (l_1X)^{\varepsilon}.
\end{split}
\end{align*}

   This implies that
\begin{align*}
%%\label{eq:mainterm}
\begin{split}
S_{10}(l) =&  \frac{
X}{\zeta^{(2)}(2)\sqrt{l_1}}\sum_{\substack{m=1 \\ (m,2)=1 }}^\infty\frac{\sigma_f(l_1m^2)}{m} \widetilde{\Phi}_{l_1m^2}(0) \prod_{p\mid lm}\frac{p}{p+1} + O\left( \frac{X(l_1X)^{\varepsilon}}{Z\sqrt{l_1}}\right).
\end{split}
\end{align*}

  We then deduce from \eqref{phi0} that
\begin{equation}
\label{eqn:4.4}
S_{10}(h)= \frac{
X}{\zeta^{(2)}(2)\sqrt{l_1}}\int\limits_{(c)}w(s) \Big( \frac{X^{3/2}}{l_1} \Big)^s \mathcal Z(1+2s,l)\widehat{\Phi}(\tfrac{3s}{2}+1) \frac{ W(s) \dif s}{s}+ O\left( \frac{X(l_1X)^{\varepsilon}}{Z\sqrt{l_1}}\right),
\end{equation}
   where $\mathcal Z(s,l)$ is defined in \eqref{Zdef}. \newline

   Note that $\zeta(s)$, $L(s,f)$ and $L(s, \sym f)$ are $L$-functions of degree $1$, $2$ and $3$, respectively. The convexity bound (see \cite[exercise 3, p. 100]{iwakow}) then implies that for $0< \Re(s)<1$,
\begin{align*}
%%\label{zetabound}
   \zeta(s) L(s,\sym f)L(s, f)\ll (1+|s|^6)^{(1-\Re(s))/2+\varepsilon}.
\end{align*}
  This together with the rapid decay of $w(s)$ on vertical strips given in \eqref{wsdecay} and the estimation that
  \[ |\prod\limits_{p|l}Q_p(s;l)| \ll\ l^{\varepsilon}\]
   for $\Re(s) > 1/2+\varepsilon$ in Lemma \ref{lemma:DS4P} allows us to see that upon moving the line of integration in \eqref{eqn:4.4} to $\Re(s) = -1/4+\varepsilon$, the contribution of the integration on the new line is
\begin{align}
\label{EforS10}
  \ll \frac {X^{5/8+\varepsilon}l^{\varepsilon}}{\sqrt{l_1}}.
\end{align}

  Also, we encounter a double pole at $s=0$ by Lemma \ref{lemma:DS4P}, whose residue contributes
\begin{align*}
%%\label{ResS10}
\begin{split}
  \frac{X \widehat{\Phi}(1)}{4\zeta^{(2)}(2)\sqrt{l_1}} \prod\limits_{p\mid l}\Big ( \frac{p}{p+1} \Big )L^{(2)}(1,\sym f)L^{(2)}(1, f)\mathcal H(1)\prod_{p|l}Q_p(1;l) \Big (\log \frac {X^{3/2} }{l_1}+C_1+ 2\sum_{p|l}\frac {Q'_p(1;l)}{Q_p(1;l)} \Big ),
\end{split}
\end{align*}
  where $C_1$ is a constant depending on $\Phi(1)$ and $\Phi'(1)$ only. \newline

  We now apply Lemma \ref{lemma:DS4P} to see that the above expression can be written as
\begin{align}
\label{ResS10}
\begin{split}
  \frac{X \widehat{\Phi}(1)}{4\zeta^{(2)}(2)\sqrt{l_1}} \prod\limits_{p\mid l}\Big ( \frac{p}{p+1} \Big ) & L^{(2)}(1,\sym f)L^{(2)}(1, f)\mathcal H(1)\prod_{p|l}Q_p(1;l) \\
& \hspace{0.5in} \times  \Big (\log \frac {X^{3/2} }{l_1}+C_1-2\sum_{p|l_1}\frac { \log p/p}{1+\lambda_f(p)+1/p}+ 2\sum_{p|l_1}\frac {D_1(p) \log p}{p} \Big ),
\end{split}
\end{align}
 where $D_1(p) \ll 1$. \newline

  We conclude from \eqref{eqn:4.4}-\eqref{ResS10} that
\begin{align} \label{S10}
\begin{split}
S_{10}(h)
 =  & \frac{X \widehat{\Phi}(1)}{4\zeta^{(2)}(2)\sqrt{l_1}} \prod\limits_{p\mid l}\Big ( \frac{p}{p+1} \Big )L^{(2)}(1,\sym f)L^{(2)}(1, f)\mathcal H(1)\prod_{p|l}Q_p(1;l) \\
& \hspace*{1cm} \times \Big (\log \frac {X^{3/2} }{l_1}+C_1-2\sum_{p|l_1}\frac { \log p/p}{1+\lambda_f(p)+1/p}+ 2\sum_{p|l_1}\frac {D_1(p) \log p}{p} \Big ) + O\left(  \frac{X(l_1X)^{\varepsilon}}{Z\sqrt{l_1}}+\frac {X^{5/8+\varepsilon}l^{\varepsilon}}{\sqrt{l_1}} \right).
\end{split}
\end{align}

\subsection{Estimating $S_1(l)$, the $k \neq 0$ terms}
\label{section:3.3}
Let $S_3(l)$ denote the contribution to $S_1(l)$ of the terms  $k \neq 0$ in \eqref{eq:S1} so that
\begin{align}
\label{eq:S3}
 S_3(l)= X \sum_{n} \frac {\sigma_f(n)}{n^{1/2}} \frac {1}{nl}\left( \frac {16}{nl} \right)
   \sum_{\substack{a \leq Z \\ (a,2nl)=1}} \frac{\mu(a)}{a^2} \sum_{k \neq 0}(-1)^kG_k(nl)\widetilde{\Phi}_n\left( \frac {kX}{2a^2 nl} \right),
\end{align}
 where
\begin{align*}
%%\label{eq:S3}
\begin{split}
  \widetilde{\Phi}_n\left(y \right)=& \int\limits^{\infty}_{0}\left(\cos +\sin \right) (2\pi  xy) \Phi_n(x) \dif x=\int\limits^{\infty}_{0}\left(\cos +\sin  \right) (2\pi  xy)\Phi(x)V \Big( \frac n{(Xx)^{3/2}} \Big) \dif x \\
=& \frac{1}{2 \pi i} \int\limits^{\infty}_{0}\left(\cos +\sin  \right) (2\pi  xy)\Phi(x) \int\limits\limits_{(2)} \frac{W(s)}{s} w(s) \Big (\frac {(Xx)^{3/2}}{n} \Big )^{s} \dif s \dif x.
\end{split}
\end{align*}

  Further note that by Mellin inversion, we have
\begin{align*}
%%\label{eq:S3}
  \Phi(x) = \frac{1}{2 \pi i} \int\limits\limits_{(2)} \widehat\Phi(1+u)x^{-u-1} \dif u.
\end{align*}

  It follows that
\begin{align*}
%%\label{eq:S3}
\begin{split}
  \widetilde{\Phi}_n\left(y \right)=& \Big (\frac{1}{2 \pi i}\Big )^2
\int\limits^{\infty}_{0}\left(\cos +\sin  \right) (2\pi  x) \int\limits\limits_{(2)}\int\limits\limits_{(2)} \frac{W(s)}{s} w(s) \Big (\frac {(Xx)^{3/2}}{n} \Big )^{s} \widehat\Phi(1+u)x^{-u-1} \dif s \dif u  \dif x.
\end{split}
\end{align*}

  Upon a change of variable $x \longrightarrow  x/(2\pi |y|)$, we recast the above as
\begin{align*}
%%\label{eq:S3}
\begin{split}
  \widetilde{\Phi}_n\left(y \right)=&
\Big (\frac{1}{2 \pi i}\Big )^2 \int\limits^{\infty}_{0}\left(\cos (x)+\text{sgn}(y) \sin(x) \right)  \int\limits\limits_{(2)}\int\limits\limits_{(2)} \frac{W(s)}{s} w(s) \Big (\frac {Xx}{2\pi |y|}\Big )^{3s/2}n^{-s} \widehat\Phi(1+u)\big (\frac {x}{2\pi |y|}\big )^{-u} \dif s \dif u \frac {\dif x}{x} .
\end{split}
\end{align*}

For $0<\Re (w)<1$, we have (see \cite[(5.8)]{Young2})
\begin{align*}
%%\label{eq:S3}
\begin{split}
  \int\limits^{\infty}_0\cos(x)x^w\frac {\dif x}{x}=\Gamma(w)\cos \big( \frac {\pi w}{2} \big),
\end{split}
\end{align*}
  and this identity also holds with $\cos x$ replaced by $\sin x$. It follows that
\begin{align*}
%%\label{phinexp}
\begin{split}
  \widetilde{\Phi}_n\left(y \right)=&
  \Big (\frac{1}{2 \pi i}\Big )^2 \int\limits\limits_{(2)}\int\limits\limits_{(2)} \frac{W(s)}{s n^s} w(s) X^{3s/2} \widehat\Phi(1+u)\Big (\frac {1}{2\pi |y|} \Big)^{3s/2-u}\Gamma \Big( \frac {3s}{2}-u \Big)\left(\cos+\text{sgn}(y) \sin \right) \Big( \frac {\pi}{2}(\frac {3s}{2}-u) \Big)  \dif s \dif u.
\end{split}
\end{align*}

   Applying the above expression in \eqref{eq:S3} allows us to deduce
that
\begin{align*}
%%\label{eq:S31}
\begin{split}
 S_3(l)
=  X \sum_{\substack{a \leq Z \\ (a,2l)=1}} \frac{\mu(a)}{a^2}\Big (\frac{1}{2 \pi i}\Big )^2 \int\limits\limits_{(2)}\int\limits\limits_{(2)} & \frac{W(s)}{s} w(s) X^{u} \widehat\Phi(1+u)\Big (\frac {\alpha^2 l}{\pi}\Big )^{3s/2-u}\Gamma \Big(\frac {3s}{2}-u\Big) \left(\cos +\text{sgn}(k)\sin \right) \Big(\frac {\pi}{2} \Big(\frac {3s}{2}-u \Big) \Big) \\
& \times \sum_{k \neq 0} \sum_{(n, 2a)=1} \frac {\sigma_f(n)}{n^{1/2+u-s/2}|k|^{3s/2-u}}  \frac {(-1)^kG_{k}(nl)}{nl}   \dif s \dif u.
\end{split}
\end{align*}

    Now, we write $k = \iota k_1k^2_2$ with $\iota \in \{ \pm 1 \}$ and $k_1>0$ being square-free.  Set $f(k)=G_{\iota k}(nl)/|k|^{z}$, where $z=3s/2-u$. It follows from  \cite[(5.15)]{Young2} that
\begin{align*}
   \sum^{\infty}_{k=1}(-1)^{k} f(k)= (2^{1-2z}-1)\sumstar_{\substack{k_1  \geq 1 \\ (k_1, 2) = 1}}
   \sum^{\infty}_{k_2=1}f(k_1k^2_2) +\sumstar_{\substack{k_1 \geq 1 \\ 2| k_1}}\sum^{\infty}_{k_2=1}f(k_1k^2_2).
\end{align*}

    We apply the above relation to recast $S_3(l)$ as
\begin{align}
\label{S3exp}
     S_3(l) &=   X \sum_{\substack{a \leq Z \\ (a,2l)=1}} \frac{\mu(a)}{a^2}
\sum_{\iota =\pm 1} \left (  \ \sumstar_{\substack{k_1 \geq 1 \\ (k_1, 2) = 1}} \frac {\mathcal{M}^{\iota}_{1}(s,u,k_1,l)}{k_1^{3s/2-u}}
+ \sumstar_{\substack{k_1 \geq 1 \\ 2|k_1 }}\frac {\mathcal{M}^{\iota}_{2}(s,u, k_1,l)}{k_1^{3s/2-u}} \right ),
\end{align}
    where
\begin{align}
\label{eq:M1exp}
\begin{split}
 \mathcal{M}^{\iota}_{1}(s,u, k_1,l) = \Big (\frac{1}{2 \pi i}\Big )^2 \int\limits\limits_{(2)}\int\limits\limits_{(2)} & \frac{W(s)}{s} w(s) X^{u} \widehat\Phi(1+u)\Big (\frac {\alpha^2 l}{\pi k_1}\Big )^{3s/2-u}\Gamma \Big( \frac {3s}{2}-u \Big)(2^{1-3s+2u}-1)  \\
& \times \left(\cos +\iota \sin \right)\Big( \frac {\pi}{2} \Big( \frac {3s}{2}-u \Big) \Big)  J_{\iota k_1} \Big(3s-2u,\frac{1}{2}+u-\frac{s}{2} \Big) \dif s \dif u,
\end{split}
\end{align}
  with $J_{\iota k_1}$ defined in \eqref{eq:Jdef}.  The expression for $\mathcal{M}^{\iota}_{2}(s,u,k_1,l)$ is identical to that of  $\mathcal{M}^{\iota}_{1}(s,u, k_1,l)$ given in  \eqref{eq:M1exp} except that the factor $2^{1-3s+2u}-1$ is omitted. \newline

  We now apply Lemma \ref{lemma:Jprop} to see that
\begin{align*}
%%\label{eq:M1exp1}
\begin{split}
 \mathcal{M}^{\iota}_{1}(s,u, k_1,l) = \Big (\frac{1}{2 \pi i}\Big )^2 \int\limits\limits_{(2)}\int\limits\limits_{(2)} & \frac{W(s)}{s} w(s) X^{u} \widehat\Phi(1+u) \Big (\frac {\alpha^2 l}{\pi k_1}\Big )^{3s/2-u}\Gamma \Big( \frac {3s}{2}-u \Big) \\
& \times \left(\cos +\iota \sin\right) \Big( \frac {\pi}{2} \Big( \frac {3s}{2}-u \Big) \Big) \zeta(3s-2u)L^{(2)} \Big(1+u-\frac{s}{2}, \chi_{\iota k_1} \Big) L^{(2)} \Big(1+u-\frac{s}{2}, f \otimes \chi_{\iota k_1} \Big) \\
& \times K_{\iota k_1} \Big( 3s-2u, \frac{1}{2}+u-\frac{s}{2} \Big) \prod_{p|al}I_{\iota k_1,p} \Big( 3s-2u, \frac{1}{2}+u-\frac{s}{2} \Big) \dif s \dif u.
\end{split}
\end{align*}

  We move the contours to $c_s = 1/2 + \varepsilon$ and $c_u =1/4+\varepsilon$ such that $1+c_u-c_s/2 = 1+\varepsilon/2, 3c_s-2c_u=1+\varepsilon, 1/2+c_u-c_s/2 = 1/2+\varepsilon/2$.  Mindful of Lemma \ref{lemma:Jprop},  we encounter no pole in this process and $K_{\iota k_1}$ remains analytic in the process.  Next, we move $c_u$ to $-1/4+\varepsilon$ to cross a simple pole of the Dirichlet $L$-function $L^{(2)}(1+u-s/2, \chi_{\iota k_1})$ at $u = s/2$ for $\iota = k_1 = 1$ only. \newline

   Note that integration by parts implies that for any integer $E \geq 0$,
\begin{align}
\label{h1bound}
 \widehat \Phi(w) \ll  \frac{1}{(1+|w|)^{E}}.
\end{align}

Recall the estimate (see \cite[p. 1107]{S&Y}) that
\begin{equation*}
%%\label{Gammabound}
 |\Gamma(s) (\cos+\iota \sin)(\frac {\pi s}{2})| \ll |s|^{\Re(s)- 1/2}.
\end{equation*}

  It follows from the above, \eqref{h1bound} and Lemma \ref{lem:HB} that by arguing as in Section 5.3 of \cite{Young2} that the sum over $k_1$ converges absolutely on the new lines of integration and that with our choices of $c_u$ and $c_s$, the contribution to $S_3(l)$ from these lines of integration is
\begin{align}
\label{Errknonzero}
  \ll \sum_{a \leq Z} a^{-2} (l a^2)^{1 + \varepsilon} l^{-1/2 + \varepsilon} X^{3/4 + \varepsilon}
\ll l^{1/2 + \varepsilon} Z X^{3/4 + \varepsilon}.
\end{align}

  On the other hand, the residue of the pole contributes
\begin{align}
\label{residue-firstexp}
    &   \frac{X}{2 \pi i} \sum_{\substack{a \leq Z \\ (a,2l)=1}} \frac{\mu(a)}{a^2}
\int\limits\limits_{(c_s)} \frac{W(s)}{s} w(s) X^{s/2} \widehat\Phi(1+\frac s2)\Big (\frac {\alpha^2 l}{\pi}\Big )^{s}\Gamma (s)(2^{1-2s}-1)\left(\cos + \sin \right)\Big( \frac {\pi s}{2} \Big)\Res\limits_{w=1/2}J_{1}(2s, w) \dif s.
\end{align}

  We note that  \cite[Lemma 6.1]{Young2} shows that
\begin{align*}
%%\label{M'}
    \Gamma (s)(2^{1-2s}-1)\left(\cos + \sin \right)\Big( \frac {\pi s}{2}\Big)  \pi^{-s}=2 \Big( \frac {8}{\pi} \Big)^{-s}\frac {\Gamma (\frac {1/2-s}{2})}{\Gamma (\frac {1/2+s}{2})}\frac {\zeta^{(2)}(1-2s)}{\zeta(2s)}.
\end{align*}

  This together with \eqref{eq:Jk1pole} allow us to recast the expression in \eqref{residue-firstexp} as
\begin{align}
\label{residue}
    &   \frac{X}{ 2 \pi i} \sum_{\substack{a \leq Z \\ (a,2l)=1}} \frac{\mu(a)}{a^2}
\int\limits\limits_{(c_s)} \frac{W(s)}{s} w(s) X^{s/2} \widehat\Phi \Big(1+\frac s2 \Big)(\alpha^2 l)^{s} \Big( \frac {8}{\pi} \Big)^{-s}\frac {\Gamma (\frac {1/2-s}{2})}{\Gamma (\frac {1/2+s}{2})}\zeta^{(2)}(1-2s) L^{(2)}(1, f )K_{1}\Big(2s, \frac 12\Big)\prod_{p|al}I_{1,p} \Big(2s, \frac 12 \Big) \dif s.
\end{align}

   We now shift the line of integration above to $c_s=\varepsilon$, encountering no pole in the process in view of \eqref{Jgen}. By \eqref{eq:Jk1def}, we see that \eqref{residue} equals
\begin{align}
\label{residue1}
\begin{split}
    &    \frac{X}{ 2 \pi i} \sum_{\substack{(a,2l)=1}} \frac{\mu(a)}{a^2}
\int\limits\limits_{(\varepsilon)} \frac{W(s)}{s} w(s) X^{s/2} \widehat\Phi \Big(1+\frac s2 \Big)(\alpha^2 l)^{s} \Big( \frac {8}{\pi} \Big)^{-s} \frac {\Gamma (\frac {1/2-s}{2})}{\Gamma (\frac {1/2+s}{2})}\zeta^{(2)}(1-2s) L^{(2)}(1, f )K_{1} \Big( 2s, \frac 12 \Big)\prod_{p|al}I_{1,p} \Big(2s, \frac 12\Big) \dif s \\
& \hspace*{3in} +O(X^{1+\varepsilon}l^{\varepsilon}Z^{-1+\varepsilon}) \\
=& \frac{X}{2 \pi i}\int\limits\limits_{(\varepsilon )} \frac{W(s)}{s} w(s) X^{s/2} \widehat\Phi \Big(1+\frac s2 \Big)l^{s} \Big(\frac {8}{\pi}\Big)^{-s}\frac {\Gamma (\frac {1/2-s}{2})}{\Gamma (\frac {1/2+s}{2})}\zeta^{(2)}(1-2s) L^{(2)}(1, f )\mathcal F(2s;l) \dif s+O(X^{1+\varepsilon}l^{\varepsilon}Z^{-1+\varepsilon}),
\end{split}
\end{align}
with $\mathcal F(2s;l)$ defined in \eqref{Fvl}. \newline

   We further shift the line of integration on the right-hand side of \eqref{residue1} to $\Re(s)=-1/2+\varepsilon$ to pick up a double pole at $s=0$.  Note that \eqref{Edef} and \eqref{Fgenest} imply that
\begin{align*}
%%\label{residue3}
   \prod_{p|l}E_p(-1+2\varepsilon, \tfrac 12) \ll &  l^{1/2+\varepsilon}.
\end{align*}

  We apply the above together with the rapid decay of $w(s)$ on vertical lines given in \eqref{wsdecay} to see that the integral on the new line is
\begin{align*}
%%\label{residue3}
    \ll &  X^{-1/4+\varepsilon}l^{\varepsilon}.
\end{align*}

Consequently, the contribution of the integration on the new line to the expression in \eqref{residue} is
\begin{align}
\label{residueerror}
   \ll &     X^{3/4+\varepsilon}l^{\varepsilon}.
\end{align}

   We now compute the residue of the double pole at $s=0$  to see that it equals to
\begin{align*}
%%\label{residueexplicit}
\begin{split}
  &  -\frac{X \widehat\Phi(1)}{4}L^{(2)}(1, f)\mathcal F(0;l)\Big ( \log X^{1/2}+\log l+C_2+2\frac {\mathcal F'(0;l)}{\mathcal F(0;l)}  \Big ),
\end{split}
\end{align*}
  where $C_2$ is a constant depending on $\Phi(1)$ and $\Phi'(1)$ only. \newline

  We next apply \eqref{Edef} to see that the above equals
\begin{align}
\label{residueexplicit1}
\begin{split}
  &  -\frac{X \widehat\Phi(1)}{4}L^{(2)}(1, f)\mathcal F(0;l)\Big ( \log X^{1/2}+\sum_{p \mid l_1}\log p+C_2-2\sum_{\substack{p|l_1 }}\frac {(\lambda_f(p)+1/p)\log p}{1+\lambda_f(p)+1/p} +\sum_{\substack{p|l }}\frac {D_2(p) \log p}{p} \Big ),
\end{split}
\end{align}
  where $D_2(p) \ll 1$ for all $p$.

\subsection{Completion of the proof}
  We deduce from \eqref{S3exp}, \eqref{Errknonzero}, \eqref{residue1}-\eqref{residueexplicit1} that
\begin{align*}
%%\label{residueexplicit2}
\begin{split}
  S_3(l) =   -\frac{X \widehat\Phi(1)}{4}L^{(2)}(1, f)\mathcal F(0;l) & \Big ( \log X^{1/2}+\sum_{p \mid l_1}\log p+C_2-2\sum_{\substack{p|l_1 }}\frac {(\lambda_f(p)+1/p)\log p}{1+\lambda_f(p)+1/p} +\sum_{\substack{p|l }}\frac {D_2(p) \log p}{p}\Big ) \\
& +O(X^{3/4+\varepsilon}l^{\varepsilon}+l^{1/2 + \varepsilon} Z X^{3/4 + \varepsilon}+X^{1+\varepsilon}l^{\varepsilon}Z^{-1}).
\end{split}
\end{align*}
  From the expression of $\mathcal F(0;l)$ in \eqref{sumovera}, we see that the above expression for $S_3(l)$ can be written as
\begin{align}
\label{residueexplicit3}
\begin{split}
  S_3(l)=   -\frac{X \widehat{\Phi}(1)}{4\zeta^{(2)}(2)\sqrt{l_1}} & \prod\limits_{p\mid l}\Big ( \frac{p}{p+1} \Big )L^{(2)}(1,\sym f)L^{(2)}(1, f)\mathcal H(1)\prod_{p|l}Q_p(1;l) \\
& \times \Big ( \log X^{1/2}+\sum_{p \mid l_1}\log p+C_2-2\sum_{\substack{p|l_1 }}\frac {(\lambda_f(p)+1/p)\log p}{1+\lambda_f(p)+1/p} +\sum_{\substack{p|l }}\frac {D_2(p) \log p}{p} \Big ) \\
& +O(X^{3/4+\varepsilon}l^{\varepsilon}+l^{1/2 + \varepsilon} Z X^{3/4 + \varepsilon}+X^{1+\varepsilon}l^{\varepsilon}Z^{-1}).
\end{split}
\end{align}

  We put together \eqref{S2}, \eqref{S10} and \eqref{residueexplicit3} to obtain that
\begin{align}
\label{Sexp0}
\begin{split}
 S(l) = \frac{X \widehat{\Phi}(1)}{4\zeta^{(2)}(2)\sqrt{l_1}} \prod\limits_{p\mid l}\Big ( \frac{p}{p+1} \Big ) & L^{(2)}(1,\sym f)L^{(2)}(1, f)\mathcal H(1)\prod_{p|l}Q_p(1;l) \\
& \times \Big (\log \frac {X}{l_1}-\sum_{p \mid l_1}\log p+C+ 2\sum_{\substack{p|l_1 }}\frac {\lambda_f(p)\log p}{1+\lambda_f(p)+1/p}+\sum_{\substack{p|l }}\frac {D(p) \log p}{p}  \Big ) \\
&+ O\left(  \frac{X(l_1X)^{\varepsilon}}{Z\sqrt{l_1}}+\frac {X^{5/8+\varepsilon}l^{\varepsilon}}{\sqrt{l_1}}+X^{3/4+\varepsilon}l^{\varepsilon}+l^{1/2 + \varepsilon} Z X^{3/4 + \varepsilon}+\frac{X^{1+\varepsilon}l^{\varepsilon}}{Z} \right).
\end{split}
\end{align}
 where we set $C=C_1-C_2$, $D(p)=D_1(p)-D_2(p)$ (resp. $-D_2(p)$) when $p|l_1$ (resp. $p|l, p \nmid l_1$). \newline

 We now get \eqref{eq:1stmoment} from \eqref{Sexp0} by setting $Z=X^{1/8}l^{-1/4}$, thus completing the proof of Theorem \ref{theo:1stmoment}.

\vspace*{.5cm}

\noindent{\bf Acknowledgments.}  P. G. is supported in part by NSFC grant 11871082 and L. Z. by the FRG grant PS43707 at the University of New South Wales.  The authors would also like to thank H. M. Bui and G. Maiti for their helpful comments and suggestions.

\bibliography{biblio}
\bibliographystyle{amsxport}

\end{document}